\documentclass[12pt]{l4dc2021} 

\makeatletter
\let\Ginclude@graphics\@org@Ginclude@graphics 
\makeatother

\usepackage{algorithmic}
\usepackage{bm}
\usepackage{latexsym, amsfonts}
\usepackage{booktabs}
\usepackage{longtable}
\allowdisplaybreaks{}
\usepackage{tikz}
\usetikzlibrary{graphs}
\ifnum\pdfshellescape=1
\fi
\usetikzlibrary{positioning}

\newcommand{\I}{\mathbf{I}}
\newcommand{\0}{\mathbf{0}}

\title{Non-conservative Design of Robust Tracking Controllers Based on Input-output Data}
\usepackage{times}

\author{%
	\Name{Liang Xu} \Email{liang.xu@epfl.ch}\\
	\Name{Mustafa Sahin Turan} \Email{mustafa.turan@epfl.ch}\\
	\Name{Baiwei Guo} \Email{baiwei.guo@epfl.ch}\\
	\Name{Giancarlo Ferrari-Trecate} \Email{giancarlo.ferraritrecate@epfl.ch}\\
	\addr Institute of Mechanical Engineering (IGM), \'Ecole Polytechnique F\'ed\'erale de Lausanne (EPFL), Switzerland
}

\begin{document}
	
	\maketitle

	\begin{abstract}
		This paper studies worst-case robust optimal tracking using noisy input-output data.
  We utilize  behavioral system theory to represent system trajectories, while avoiding explicit system identification.
  We assume that the recent output data used to implicitly specify the initial condition are noisy and we provide a non-conservative design procedure for robust control based on optimization  with a linear cost and linear matrix inequality (LMI) constraints.
  Our methods rely on the parameterization of noise sequences compatible with the data-dependent system representation and on a suitable reformulation of the performance specification, which further enable the application of the S-lemma to derive an LMI optimization problem.
  The performance of the new controller is discussed through simulations.
\end{abstract}

	\begin{keywords}
		data-driven control, robust control, reference tracking, linear matrix inequalities.
	\end{keywords}
	
	\section{Introduction}\label{sec.Introduction}
        Due to the recent advances in pervasive sensing, communication and computation, data availability for control design is steadily increasing.
        This has motivated a renewed interest in the development of frameworks for data-driven control with performance guarantees using finite-length data sequences~\citep{DePersis2020TAC, vanWaardeHenkJ2020TAC, MatniNikolai2019CDCLearningControlSurvey, StephenTu2019Thesis}.
Several recent works use raw data for representing the system dynamics, as well as conducting system analysis and control design~\citep{DePersis2020TAC, vanWaardeHenkJ2020TAC,  BerberichJulian2020RobustDataDriven, BisoffiAndrea2020RobustInvariance, vanWaardeHenk2020MatrixSLemma, DePersis2020LowComplexity, CoulsonJeremy2019ECCDeePC, BerberichJulian2020DeeMPCStability, CoulsonJeremy2020DistributionallyRobustDeePC}.
However, most of these approaches are conceived for noiseless data or noisy input-state data.
Noisy input-output measurements are considered in~\citep{BerberichJulian2020DeeMPCStability, CoulsonJeremy2020DistributionallyRobustDeePC, RN11435}.
In~\cite{BerberichJulian2020DeeMPCStability}, slack variables are introduced in the data-dependent system representation to account for  noisy measurements.
The modified control scheme is shown to be recursively feasible and practically exponentially stable; however, the tracking performance is not analyzed.
The authors in~\cite{CoulsonJeremy2020DistributionallyRobustDeePC} propose a distributionally robust variant of DeePC based on semi-infinite optimization.
They then formulate a finite and convex program, whose optimal value is an upper bound to that of the original optimization problem.
The work~\cite{RN11435} considers using noiseless historical data and noisy recent output data to minimize the energy of the control input while robustly satisfying input/output constraints.
The authors propose to separate the problems of estimation of the initial condition and control design, and show that the solution to the formulated problem is computed by consecutively solving two optimization problems.

In safety-critical applications, such as power networks and industrial control systems, it is sometimes required to adopt a bounded-error perspective by enforcing robustness against all possible noise realizations and providing worst-case performance guarantees.
This is the setting considered in the present paper and, for this purpose, we utilize the data-driven prediction method in~\cite{MarkovskyIvan2008DataDrivSimCont}.
We assume the historical data are noiseless while recent data are corrupted by noise terms satisfying a quadratic constraint similar to the one in~\cite{vanWaardeHenk2020MatrixSLemma}.
Our goal is to provide a control design method for worst-case optimal reference tracking with explicit performance guarantees.

We first characterize noises that are consistent with the input-output data, and then reformulate the tracking cost.
This enables us to apply the S-lemma~\citep{polik2007survey} to transform the worst-case robust control problem  to an equivalent minimization problem with a linear cost and linear matrix inequality (LMI) constraints.
In contrast to~\cite{RN11435}, we aim to minimize a quadratic cost on both inputs and outputs, while the method in~\cite{RN11435} only deals with the minimization of the input energy.
The main features of our method are the following: (1) we consider the minimization of the worst-case tracking performance; (2) the proposed method does not require system identification; (3) the proposed design procedure is non-conservative, meaning that we obtain the optimal tracking controllers without any approximations.

This paper is organized as follows. In Section~\ref{sec.PreliminariesDataDriven}, we provide preliminaries on data-driven simulation and control.
The problem formulation is given in Section~\ref{sec.ProblemFormulation}.
The data-based robust optimal tracking control problem is solved in Section~\ref{sec.RobustControlDesign}.
Simulations are provided in Section~\ref{subsec.Simulation}.
Some concluding remarks are provided in Section~\ref{sec.Conclusion}.

\textbf{Notation}:
For a square matrix $\Phi$, $\Phi>0$ ($\Phi \ge 0$) represents that it is positive definite (semidefinite).
For $Q\ge 0$, the norm $\|x\|_Q$ is defined as $\sqrt{x^\top Qx}$.
For a matrix $A \in \mathbb{R}^{n\times m}$, $\ker(A)$ and $\mathrm{range}(A)$ denote its null space and column space, respectively. 
Moreover, $\mathcal{N}(A)\in \mathbb{R}^{m\times \dim(\ker(A))}$ denotes a matrix whose columns form a basis for the null space of $A$. 
$\I$ and $\0$ denote identity and zero matrices of suitable size.
The operator $\otimes$ denotes the Kronecker product.

	\section{Preliminaries on Data-driven Prediction}\label{sec.PreliminariesDataDriven}
	
	We consider a controllable discrete-time LTI system $\mathcal{G}$ with the state space model 
	\begin{equation}\label{eq.LTI_system_ss} 
	\begin{aligned}
	x_{k+1} =A x_{k}+B u_{k}, \quad	y_{k} =C x_{k}+D u_{k},
	\end{aligned}
	\end{equation}
	where $x_k, x_0\in \mathbb{R}^n, u_k\in \mathbb{R}^m, y_k\in \mathbb{R}^p$ are the system states, initial state, inputs, and outputs, respectively. Here we assume that $(A,B,C,D)$ is in minimal form; thus, the pair $(A,B)$ is controllable and $(A,C)$ is observable. The lag $\mathbf{l}(\mathcal{G})$ is defined as the smallest integer $l$ such that the $l$-step observability matrix $[C^\top,A^\top C^\top,\dots,(A^{l-1})^\top C^\top]^\top$ has rank $n$. Therefore, $\mathbf{l}(\mathcal{G})\leq n$. A system can only generate certain input-output trajectories.
		\begin{definition}
		An input-output sequence ${\{u_{k}, y_{k}\}}_{k=0}^{T-1}$ is a trajectory of $\mathcal{G}$ if and only if there exists an initial condition ${x_0} \in \mathbb{R}^{n}$ as well as a state sequence ${\{x_{k}\}}_{k=1}^{T}$ such that~\eqref{eq.LTI_system_ss} holds
		for $k=0, \ldots, T-1$.
	\end{definition}

	As common in data-driven control, to determine system characteristics, one needs to collect a set of input-output data ${\{\bar{u}_{k}, \bar{y}_{k}\}}_{k=t_h}^{t_h+T_d-1}$, which we call \textit{historical}. Historical data can be thought as collected long before the start (indicated by time $0$) of any control or prediction, i.e., $t_h\ll 0$. In the remainder of this section, we will show how to use the historical data to form a data-dependent representation of the system~\citep{WillemsJan2005SCL, MarkovskyIvan2008DataDrivSimCont}.
	Throughout the paper, the column concatenation of the vectors in a sequence ${\{v_k\}}_{k=i}^j$ is abbreviated as $v$, where the starting and ending indices $i, j$ are clear from the context.
	The Hankel matrix of \textit{depth} $L$ associated with a historical sequence ${\{\bar{v}_k\}}_{k=t_h}^{t_h+T_d-1}$ is defined as
	\begin{align*}
		\mathcal{H}_{L}(\bar{v}):=\left[\begin{array}{cccc}\bar{v}_{t_h} & \bar{v}_{t_h+1} & \cdots & \bar{v}_{t_h+T_d-L} \\ \bar{v}_{t_h+1} & \bar{v}_{t_h+2} & \cdots & \bar{v}_{t_h+T_d-L+1} \\ \vdots & \vdots & \ddots & \vdots \\ \bar{v}_{t_h+L-1} & \bar{v}_{t_h+L} & \cdots & \bar{v}_{t_h+T_d-1}\end{array}\right].
	\end{align*}
	
    To assess if $(\mathcal{H}_{L}(\bar{u})$, $\mathcal{H}_{L}(\bar{y}))$ is informative for predicting system trajectories, we introduce the concept of persistent excitation.
	\begin{definition}
		An input sequence ${\{\bar{u}_{k}\}}_{k=t_h}^{t_h+T_d-1}$ is persistently exciting of order $L$ if the corresponding Hankel matrix is full row rank, i.e., rank$\left(\mathcal{H}_{L}(\bar{u})\right)=m L$.
	\end{definition}
	The Fundamental Lemma shows how to directly use the known input-output data to characterize all possible system trajectories.
	\begin{lemma}[Fundamental Lemma~\citep{WillemsJan2005SCL}]\label{lem.WillemFundamentalLemma}
		Suppose ${\{\bar{u}_{k}, \bar{y}_{k}\}}_{k=t_h}^{t_h+T_d-1}$ is a trajectory of an LTI system $\mathcal{G}$ and that the input sequence $\{\bar{u}_{k}\}_{k=t_h}^{t_h+T_d-1}$ is persistently exciting of order $L+n$.
		Then, ${\{u_{k}, y_{k}\}}_{k=0}^{L-1}$ is a trajectory of $\mathcal{G}$ if and only if there exists $g \in \mathbb{R}^{T_d-L+1}$ such that
		\begin{align}\label{eq.WillemsLemma}
			\left[\begin{array}{c}
				\mathcal{H}_{L}(\bar{u}) \\
				\mathcal{H}_{L}(\bar{y})
			\end{array}\right]g =\left[\begin{array}{l}
				u \\
				y
			\end{array}\right].
		\end{align}
	\end{lemma}
An LTI system $\mathcal{G}$ has infinitely many trajectories corresponding to different initial states~$x_0$; therefore, Lemma~\ref{lem.WillemFundamentalLemma} cannot be directly used to predict the system output $\{y_k\}^{T_e-1}_{k=0}$ from the input sequence $\{u_k\}^{T_e-1}_{k=0}$. In order to determine the initial state and, therefore,
$\{y_k\}^{T_e-1}_{k=0}$, one also needs to know an initial trajectory $\{u_k,y_k\}^{-1}_{k=-T_{\mathrm{ini}}}$. The column concatenations of these initial sequences are denoted as $u_{\mathrm{ini}}$ and $y_{\mathrm{ini}}$, respectively. $\{u_{\mathrm{ini}}, y_{\mathrm{ini}}\}$ are measured later than the historical data $\{\bar{u}, \bar{y}\}$; therefore, we refer to the former as \textit{recent} data.The length $T_{\mathrm{ini}}$ should be no less than $\mathbf{l}(\mathcal{G})$ for $x_0$ and thus $y$ to be uniquely determined~\citep{MarkovskyIvan2008DataDrivSimCont}. When applying Lemma~\ref{lem.WillemFundamentalLemma} to characterize the system trajectory from $k=-T_{\text{ini}}$ to $k = T_e-1$, we need Hankel matrices to be of proper sizes, i.e. 
	\begin{equation*}
		U = \begin{bmatrix}
			U_p^\top &
			U_f^\top
		\end{bmatrix}^\top \triangleq \mathcal{H}_{T_{\mathrm{ini}}+T_e}\left(\bar{u}\right), \quad \quad Y=\begin{bmatrix}
			Y_p^\top &
			Y_f^\top
		\end{bmatrix}^\top \triangleq \mathcal{H}_{T_{\mathrm{ini}}+T_e}\left(\bar{y}\right),
	\end{equation*}
	where $U_p$ and $Y_p$ consist of the first $T_\mathrm{ini}$ block rows of $U$ and $Y$, while $U_f$ and $Y_f$ consist of the last $T_e$ block rows of the $U$ and $Y$. The following lemma shows how to predict the system outputs based on the Fundamental Lemma.
	\begin{lemma}[\cite{MarkovskyIvan2008DataDrivSimCont}]\label{lemma.UniqueOutput}
		Suppose $\bar{u}$ is persistently exciting of order $T_{\mathrm {ini}}+T_e+n$, and  $T_{\mathrm {ini}} \geq \mathbf{l}(\mathcal{G})$.
		Then for a system trajectory $(u_\mathrm{ini},y_\mathrm{ini})$ and any $T_e$-long input sequence $u$, the following equation
		\begin{align} \label{eq.DataDriveSimControl}
			\left[\begin{array}{l}
				U_p \\
				Y_p \\
				U_f \\
				Y_f
			\end{array}\right] g=\left[\begin{array}{c}
				u_{\mathrm{ini}} \\
				y_{\mathrm {ini}} \\
				u\\
				y
			\end{array}\right]
		\end{align}
		can be solved for $g$ and $y$, where the solution $y$ is unique.
	\end{lemma}
	
	\section{Problem Formulation}\label{sec.ProblemFormulation}
	
	Based on the above method of system simulation, with the \textit{historical} data $\{\bar{u},\bar{y}\}$ and \textit{recent} data $\{u_\mathrm{ini}, y_\mathrm{ini}\}$ at hand, we formulate the problem of data-driven linear-quadratic tracking over a finite horizon as
	\begin{align}
		\min_{u, g}\sum_{k=0}^{T_e-1}\left(\left\|y_{k}-r_{k}\right\|_{Q}^{2}+\left\|u_{k}\right\|_{R}^{2}\right) \label{eq.yini-performanceconstraint-noiseless} \quad 
		\mathrm{s.t.}~\eqref{eq.DataDriveSimControl},
	\end{align}
	where $r$ represents the output reference to be tracked; $Q$ and $R$ are positive semi-definite matrices; $u$ is the control input to be designed; $y$ is the resulting output from $u$ and also the unique solution to~\eqref{eq.DataDriveSimControl} in view of Lemma~\ref{lemma.UniqueOutput}.
	
	In this paper, we are interested in the case that initial output trajectory $y_{\mathrm{ini}}$ is noisy, i.e.,
	$y_{\mathrm{ini}}=\check{y}_{\mathrm{ini}}+w$,
	where $\check{y}_{\mathrm{ini}}$ represents the noiseless output signal and $w$ represents the measurement noise.
	Moreover, we assume that $w$ satisfies the quadratic constraint, first introduced in~\cite{vanWaardeHenk2020MatrixSLemma} and~\cite{BerberichJulian2020RobustDataDriven}
	\begin{align}
		\begin{bmatrix}
			1\\
			w
		\end{bmatrix}^\top
		\underbrace{\begin{bmatrix}
				\Phi_{11} & \Phi_{12}\\
				\Phi_{12}^\top & \Phi_{22}
		\end{bmatrix}}_{\Phi}
		\begin{bmatrix}
			1\\
			w
		\end{bmatrix}
		\ge 0, ~\text{ where } \Phi_{22}=\Phi_{22}^\top< 0. \label{eq.yini-noiseconstraint}
	\end{align}
	
	\begin{remark}\label{rem.wbounded}
		As shown in~\cite{vanWaardeHenk2020MatrixSLemma} and~\cite{BerberichJulian2020RobustDataDriven}, the negative definiteness of $\Phi_{22}$ ensures that noise $w$ is bounded.
		In the special case that $\Phi_{12}=\0$ and $\Phi_{22}=-\I$,~\eqref{eq.yini-noiseconstraint} reduces to
		$w^\top w=\sum_{i}w_i^\top w_i\le \Phi_{11},$
		which has the interpretation of bounded accumulated energy for $w$.
	\end{remark}
	
	\begin{remark}
		We assume that the historical data $\{\bar{u}_k,\bar{y}_k\}^{t_h+T_d-1}_{k=t_h}$ is not affected by noise, but only the recent output measurements $y_\mathrm{ini}$ are. This assumption is realistic as in certain practical scenarios one might have access to very accurate (and, thus, expensive) sensors to collect historical data once, but only have relatively inaccurate and noisy sensors to collect data during real-time operations.
	\end{remark}

	We are interested in designing the control input $u$ that minimizes the worst-case quadratic tracking error~\eqref{eq.yini-performanceconstraint-noiseless} among all \textit{feasible} noise trajectories $w$, i.e., the vectors $w$ satisfying~\eqref{eq.yini-noiseconstraint}, such that $(u_{\mathrm{ini}}, y_{\mathrm{ini}}-w)$ is a trajectory of $\mathcal{G}$, as per Lemma~\ref{lem.WillemFundamentalLemma}. The formal min-max robust optimal tracking control problem is given as follows.
	
	\noindent \textbf{Problem~P1}: Find the input sequence $u$ solving the min-max optimization problem
	\begin{align}
		\min_u \max_{w,g}& \sum_{k=0}^{T_e-1}\left(\left\|y_{k}-r_{k}\right\|_{Q}^{2}+\left\|u_{k}\right\|_{R}^{2}\right) \nonumber \\
		\text{s.t.,} &
		\begin{bmatrix}
			U_p \\
			Y_p \\
			U_f \\
			Y_f
		\end{bmatrix} g=\begin{bmatrix}
			u_{\mathrm{ini}} \\
			y_{\mathrm{ini}} \\
			u \\
			y
		\end{bmatrix}-\begin{bmatrix}
			\0 \\
			w \\
			\0 \\
			\0
		\end{bmatrix}, \label{eq.yini-model-minmax}\\
		&\begin{bmatrix}
			1\\
			w
		\end{bmatrix}^\top\Phi
		\begin{bmatrix}
			1\\
			w
		\end{bmatrix}
		\ge 0 \label{eq.yini-noiseconstraint-minmax}.
	\end{align}
	
	\begin{remark}
		In view of Lemma~\ref{lemma.UniqueOutput}, there is a unique output $y$ for given $u$ and $w$. Therefore, even though multiple $g$ verifying~\eqref{eq.yini-model-minmax} might exist, they are completely equivalent, since they yield the same input-output trajectory. As such, in the optimization problem \textbf{P1}, the optimization variable $g$ can be omitted for notational simplicity. 
              \end{remark}
              
	In view of Lemma~\ref{lem.WillemFundamentalLemma} and above min-max optimization problem, the \textit{feasible} noises are $w$ vectors satisfying~\eqref{eq.yini-noiseconstraint-minmax}, such that~\eqref{eq.yini-model-minmax} admits a solution $g$. In the next section, we propose a method to solve \textbf{P1}. 
	
	\section{Robust Controller Design}\label{sec.RobustControlDesign}
	Problem~\textbf{P1} can be reformulated as
	\begin{equation}\label{eq.minimization_problem}
	\begin{gathered}
	\min_{u, \gamma} \quad \gamma \\
	\text{s.t.}, \quad \text{LQTE}(u,w)\leq \gamma ~~ \forall w
	\text{ satisfying }
	\eqref{eq.yini-model-minmax},\eqref{eq.yini-noiseconstraint-minmax},
	\end{gathered}
	\end{equation}
	where the linear quadratic tracking error is defined as
	$\text{LQTE}(u,w)\triangleq \sum_{k=0}^{T_e-1}\left(\left\|y_{k}-r_{k}\right\|_{Q}^{2}+\left\|u_{k}\right\|_{R}^{2}\right)$.
	For notational simplicity, we have omitted the dependence of $\text{LQTE}$ on $r$.
	
	In the sequel, we will derive a tractable reformulation of~\eqref{eq.minimization_problem}.
	We first show in subsection~\ref{subsec.NoiseConstTransform} that any noise $w$ satisfying~\eqref{eq.yini-model-minmax} and~\eqref{eq.yini-noiseconstraint-minmax} can be parameterized by a vector $g_w$ satisfying a quadratic constraint.
	In subsection~\ref{subsec.PerformanceTransformation}, we show that the output $y$ is completely determined by the input $u$ and the vector $g_w$, which further allows us to express the constraint $\text{LQTE}(u, w)\le \gamma$ in~\eqref{eq.minimization_problem} as a quadratic constraint on $g_w$.
	In light of these results, in subsection~\ref{subsec.MainResult}, we show that~\eqref{eq.minimization_problem} is equivalent to a minimization problem with a linear cost and LMI constraints. 

	\subsection{Feasible Noise Parameterization}\label{subsec.NoiseConstTransform}
Since $\bar{u}=\{\bar{u}_k\}_{k=t_h}^{t_h+T_d-1}$ is persistently exciting of order $T_{\mathrm{ini}}+T_e+n$, we know that $\{\bar{u}_k\}_{k=t_h}^{t_h+T_d-T_e-1}$ is persistently exciting of order $T_{\mathrm{ini}}+n$. In view of Lemma~\ref{lem.WillemFundamentalLemma},  $(u_\mathrm{ini},y_\mathrm{ini}-w)$ is a trajectory of~$\mathcal{G}$ if and only if there exists a vector $g_\mathrm{ini}$, such that
	\begin{align}
		\begin{bmatrix}
			U_p\\
			Y_p
		\end{bmatrix}g_\mathrm{ini}=
		\begin{bmatrix}
			u_{\mathrm{ini}}\\
			y_{\mathrm{ini}}-w
		\end{bmatrix}. \label{eq.yini-trajectoryconstraint}
	\end{align}
        Therefore, a noise vector $w$ verifies~\eqref{eq.yini-model-minmax} if and only if it belongs to the set
	\begin{equation}
	\mathcal{W}\triangleq \left\{w\in\mathbb{R}^{pT_\mathrm{ini}}\left|\begin{bmatrix}
	u_\mathrm{ini} \\
	y_\mathrm{ini}-w
	\end{bmatrix}\in\mathrm{range}\left(\begin{bmatrix}
	U_p \\
	Y_p
	\end{bmatrix}\right)\right.\right\}.
	\end{equation}
	
	In the following lemma, we show that the set $\mathcal{W}$ can be parameterized by a vector $g_w$, and the proof can be found in~\cite{RN11538}.
	\begin{lemma}\label{lem.g_w_parameterization}
		Let $n_w = \dim(\ker(U_p))$. Then, any noise $w\in\mathcal{W}$ can be expressed as an affine function of a free vector $g_w\in\mathbb{R}^{n_w}$ as
		\begin{equation}\label{eq.w_parameterization}
		w=-Y_pMg_w + \underbrace{(- Y_pg_w^*+y_\mathrm{ini})}_{w_0},
		\end{equation}
		where $M=\mathcal{N}(U_p)$ 
		and $g_w^*=U_p^\top \left(U_pU_p^\top\right)^{-1}u_{\mathrm{ini}}$.
		Moreover, any feasible noise $w$ satisfying \eqref{eq.yini-model-minmax} and \eqref{eq.yini-noiseconstraint-minmax} can be represented as in~\eqref{eq.w_parameterization} with the additional constraint
		\begin{equation}\label{eq.g_w_constraint}
		\begin{bmatrix}
		1\\
		g_w
		\end{bmatrix}^\top
		\underbrace{\begin{bmatrix}
			\Phi_{11}+w_0^\top\Phi_{12}^\top+\Phi_{12}w_0+w_0^\top\Phi_{22}w_0 & -\Phi_{12}Y_pM-w_0^\top\Phi_{22}Y_pM \\
			-M^\top Y_p^\top \Phi_{12}^\top -M^\top Y_p^\top\Phi_{22} w_0 & M^\top Y_p^\top\Phi_{22}Y_pM
			\end{bmatrix}}_{A_w}
		\begin{bmatrix}
		1\\
		g_w
		\end{bmatrix}\geq 0.
		\end{equation}
	\end{lemma}
	

	\subsection{Transformation of the Performance Specifications}\label{subsec.PerformanceTransformation}
	
	In this subsection, we show that for given feasible $w$, the output $y$ can be expressed in terms of $g_w$ and $u$, and that the performance specification constraint LQTE$(u,w)\leq\gamma$ can be transformed into a quadratic constraint on $g_w$.
	For given $u$ and feasible $w$, we first show how to compute $g$ in~\eqref{eq.yini-model-minmax}, which can further be used to calculate the output $y$.
	When $(u_{\mathrm{ini}}, y_{\mathrm{ini}}-w)$ is a feasible initial system trajectory, for any input $u$, there exists a $g$ verifying
	\begin{align}
		\begin{bmatrix}
			U_p\\
			Y_p\\
			U_f
		\end{bmatrix}g=
		\begin{bmatrix}
			u_{\mathrm{ini}}\\
			y_{\mathrm{ini}}-w\\
			u
		\end{bmatrix}.\label{eq.yini-FirstThreeSolution}
	\end{align}
	Although $g$ is not necessarily unique, all such $g$ produce the same $y$ (see Lemma~\ref{lemma.UniqueOutput}).
	In the following, we show that a candidate solution $g$ to~\eqref{eq.yini-FirstThreeSolution} is given by $g_\mathrm{ini} + g_u$, where $g_\mathrm{ini}=g_w^*+Mg_w$ is a solution to~\eqref{eq.yini-trajectoryconstraint}  and $g_u$ is a solution to
	
	\begin{equation}\label{eq.g_u}
	\begin{bmatrix}
	U_p\\
	Y_p\\
	U_f
	\end{bmatrix}g_u=
	\begin{bmatrix}
	\mathbf{0}\\
	\mathbf{0}\\
	u-U_fg_\mathrm{ini}
	\end{bmatrix}.
	\end{equation}
	Since $[\0^\top,\0^\top]^\top$ is a feasible initial system trajectory, in  view of Lemma~\ref{lemma.UniqueOutput}, there always exists a $g_u$ solving~\eqref{eq.g_u}.
	For explicitly characterizing $g_u$, we first introduce a preliminary lemma.
	\begin{lemma}\label{lemma.Yp2LinearCombination}
		There always exists a row permutation matrix $P_Y$ decomposing $Y_p$ as $P_YY_p = [Y_{p1}^\top,Y_{p2}^\top]^\top$ such that $\Lambda \triangleq[U_p^\top,Y_{p1}^\top,U_f^\top]^\top$ has full row rank and
		$\text{rank}(\Lambda)=\text{rank}([
		U_p^\top, Y_{p}^\top,	U_f^\top
		]^\top )$.   
		For such a $P_Y$, the rows of $Y_{p2}$ can be written as linear combinations of the rows of $[U_p^\top,  Y_{p1}^\top]^\top$.  
	\end{lemma}
	
	\begin{proof}
		It is straightforward to show the existence of such a $P_Y$; therefore, the proof of this fact is omitted here.
		We apply the following row permutation to~\eqref{eq.g_u} 
		\begin{align}
			\begin{bmatrix}
				\mathbf{I} & \mathbf{0} & \mathbf{0} \\
				\mathbf{0} & P_Y & \mathbf{0} \\
				\mathbf{0} & \mathbf{0} & \mathbf{I}
			\end{bmatrix}\begin{bmatrix}
				U_p\\
				Y_{p}\\
				U_f
			\end{bmatrix}g=\begin{bmatrix}
				U_p\\
				Y_{p1}\\
				Y_{p2}\\
				U_f
			\end{bmatrix}g=
			\begin{bmatrix}
				\0\\
				\0\\
				\0\\
				u-U_fg_\mathrm{ini}
			\end{bmatrix}.
			\label{eq.yini-initialtrajectoryconstraint}
		\end{align}
		By definition, the rows of $Y_{p2}$ can be written as linear combinations of the rows of $[U_p^\top,Y_{p1}^\top,U_f^\top]^\top$. Therefore, there exists an ordered sequence of elementary row operations $\{E_k\}_{k=1}^{e}$ captured by the matrix $E\triangleq E_eE_{e-1}\dots E_1$ such that
		\begin{align*}
			E\begin{bmatrix}
				U_p^\top&
				Y_{p1}^\top&
				Y_{p2}^\top&
				U_f^\top
			\end{bmatrix}^\top=
			\begin{bmatrix}
				U_p^\top&
				Y_{p1}^\top&
				\0&
				U_f^\top
			\end{bmatrix}^\top.
		\end{align*}
		
		Suppose, by contradiction, that the rows of $Y_{p2}$ cannot be written as linear combinations of the rows of $[U_p^\top,Y_{p1}^\top]^\top$.
		Then, applying $E$ to both sides of~\eqref{eq.yini-initialtrajectoryconstraint}, we would obtain
		\begin{align}
			\begin{bmatrix}
				U_p\\
				Y_{p1}\\
				\0\\
				U_f
			\end{bmatrix}g=
			\begin{bmatrix}
				\0\\
				\0\\
				\text{linear combination of rows of }u-U_fg_\mathrm{ini}\\
				u-U_fg_\mathrm{ini}
			\end{bmatrix}.
			\label{eq: permutationcontradiction}
		\end{align}
		For any $g_{\mathrm{ini}}$, there exists a $u$ such that \eqref{eq: permutationcontradiction} does not admit a solution, which contradicts the fact that~\eqref{eq.g_u} is feasible.
	\end{proof}
	
	In the next lemma, we give an explicit formula for $g_u$ solving~\eqref{eq.g_u}.
	\begin{lemma}\label{lem.g_u_solution}
		A solution to~\eqref{eq.g_u} is provided by 
		\begin{equation}\label{eq.g_u_solution}
		g_u=\Lambda^\top (\Lambda\Lambda^\top)^{-1}\begin{bmatrix}
		\mathbf{0}\\
		\mathbf{0}\\
		u-U_fg_\mathrm{ini}
		\end{bmatrix}.
		\end{equation}
	\end{lemma}
	
	\begin{proof}
		The vector $g_u$ in~\eqref{eq.g_u_solution} satisfies $\Lambda g_u = [\0,\0,(u-U_fg_\mathrm{ini})^\top]^\top$, and, therefore, 
		%
		\begin{align} \label{eq.temp1}
			\begin{bmatrix}
				U_p\\
				Y_{p1}\\
				\0\\
				U_f
			\end{bmatrix}g_u=
			\begin{bmatrix}
				\0\\
				\0\\
				\0\\
				u-U_f g_\mathrm{ini}
			\end{bmatrix}.
		\end{align}
		In view of the proof of Lemma~\ref{lemma.Yp2LinearCombination}, there exists a matrix $E$ satisfying
		\begin{align} \label{eq:1}
			E
			\begin{bmatrix}
				U_p\\
				Y_{p1}\\
				Y_{p2}\\
				U_f
			\end{bmatrix}=
			\begin{bmatrix}
				U_p\\
				Y_{p1}\\
				\0\\
				U_f
			\end{bmatrix}~\text{and}~
			E
			\begin{bmatrix}
				\0\\
				\0\\
				\0\\
				u-U_f g_\mathrm{ini}
			\end{bmatrix}=
			\begin{bmatrix}
				\0\\
				\0\\
				\0\\
				u-U_f g_\mathrm{ini}
			\end{bmatrix}. 
		\end{align}
		Left-multiplying both sides of~\eqref{eq.temp1} by $E^{-1}$, in view of~\eqref{eq:1}, one obtains
		$[
		U_p^\top,
		Y_{p1}^\top,
		Y_{p2}^\top,
		U_f^\top
		]^\top g_u
		=
		[
		\0,
		\0,
		\0,
		(u-U_f g_\mathrm{ini})^\top
		]^\top$.
		From~\eqref{eq.yini-initialtrajectoryconstraint}, we conclude the proof by showing that
		\begin{align*}
			& \begin{bmatrix}
				U_p\\
				Y_p\\
				U_f
			\end{bmatrix}g_u= \begin{bmatrix}
				\I & \0 & \0 \\
				\0 & P_Y^{-1}& \0 \\
				\0 & \0 & \I
			\end{bmatrix}
			\begin{bmatrix}
				U_p\\
				\begin{bmatrix}
					Y_{p1}\\
					Y_{p2}            
				\end{bmatrix}\\
				U_f
			\end{bmatrix}g_u=
			\begin{bmatrix}
				\I & \0 & \0 \\
				\0 & P_Y^{-1}& \0 \\
				\0 & \0 & \I
			\end{bmatrix}
			\begin{bmatrix}
				\0\\
				\begin{bmatrix}
					\0\\
					\0
				\end{bmatrix}\\
				u-U_f g_\mathrm{ini}
			\end{bmatrix}
			=
			\begin{bmatrix}
				\0\\
				\begin{bmatrix}
					\0\\
					\0
				\end{bmatrix}\\
				u-U_f g_\mathrm{ini}
			\end{bmatrix}.
		\end{align*}
	\end{proof}
	
	Given an explicit solution for $g_u$ and, therefore, a solution $g$ for~\eqref{eq.yini-FirstThreeSolution}, we next derive the expression of $y$ and characterize the performance specification LQTE$(u,w)\leq\gamma$  in terms of $g_w$. The proof can be found in~\cite{RN11538}.
	
	\begin{lemma}\label{lem.y_LQTE_parameterization}
		Consider the matrix $\Lambda$ defined in Lemma~\ref{lemma.Yp2LinearCombination}. Given a feasible noise $w\in\mathcal{W}$ and a control sequence $u$, the unique output $y$ satisfying~\eqref{eq.yini-model-minmax} is given by 
		\begin{equation}\label{eq.y_parameterization}
		y = B_uu+B_wg_w+y_0,
		\end{equation}
		where $y_0 = B_\mathrm{ini}g_w^*$,  $B_w = B_\mathrm{ini}M$,
		\begin{align*}
			B_\mathrm{ini} = Y_f\left(\I+\Lambda^\top (\Lambda\Lambda^\top)^{-1}
			\begin{bmatrix}
				\0\\
				\0\\
				-U_f
			\end{bmatrix}\right),\quad
			B_u = Y_f \Lambda^\top (\Lambda\Lambda^\top)^{-1}
			\begin{bmatrix}
				\0\\
				\0\\
				\I
			\end{bmatrix}.
		\end{align*}
		Moreover,  the performance specification $\mathrm{LQTE}(u,w)\leq\gamma$ can be equivalently expressed as
		\begin{align}
			\begin{bmatrix}
				1 \\
				g_w
			\end{bmatrix}^\top
			\underbrace{\begin{bmatrix}
					\gamma-u^\top\bar{R}u-(B_uu+y_0-r)^\top\bar{Q}(B_uu+y_0-r) & -(B_uu+y_0-r)^\top\bar{Q}B_w \\
					-B_w^\top\bar{Q}(B_uu+y_0-r) & -B_w^\top\bar{Q}B_w
			\end{bmatrix}}_{Q_g(u,\gamma)}
			\begin{bmatrix}
				1\\
				g_w
			\end{bmatrix}\ge 0, \label{yini-performancequadraticform}
		\end{align}
		where $\bar{R}=\I\otimes R$, $\bar{Q}=\I\otimes Q$ and $\Lambda$ is defined in Lemma~\ref{lemma.Yp2LinearCombination}.
	\end{lemma}
	
	
	\subsection{Main Result} \label{subsec.MainResult}
	
	The following theorem leverages the results obtained in Lemmas~\ref{lem.g_w_parameterization} and~\ref{lem.y_LQTE_parameterization} to show that the minimization problem~\eqref{eq.minimization_problem}, and hence \textbf{P1}, are equivalent to a minimization problem with a linear cost and LMI constraints.
	
	\begin{theorem}\label{thm.LMIopt}
		The robust control problem \textbf{P1} is equivalent to solving
		\begin{gather} \label{eq.finalOpt}
			\min_{u, \gamma, \alpha\ge 0}\; \gamma\\
			\mathrm{s.t.},
			\begin{bmatrix}
				{(\bar{R}+B_u^\top\bar{Q}B_u)}^{-1}&
				\begin{bmatrix}
					u & \0      
				\end{bmatrix}\\
				\begin{bmatrix}
					u^\top\\
					\0
				\end{bmatrix}& Q_g^a(u,\gamma)-\alpha A_w
			\end{bmatrix}\ge 0,\label{eq.finalLMI}
		\end{gather}
		where
                \begin{multline*}
                  			Q_g^a(u, \gamma)=\\
                  \begin{bmatrix}
                    \gamma-(B_uu)^\top\bar{Q}(y_0-r)
                    -(y_0-r)^\top\bar{Q}(B_uu)
                    -(y_0-r)^\top\bar{Q}(y_0-r)                  
                    & -(B_uu+y_0-r)^\top\bar{Q}B_w \\
					-B_w^\top\bar{Q}(B_uu+y_0-r) & -B_w^\top\bar{Q}B_w
			\end{bmatrix}.
                \end{multline*}
			\end{theorem}
	
	\begin{proof}
		Based on the feasible noise parameterization and performance specification transformations, the minimization problem~\eqref{eq.minimization_problem} is equivalent to 
		\begin{gather*}
			\min_{u,\gamma}  \gamma \quad
			\text{s.t.},   \eqref{yini-performancequadraticform} ~\text{holds}~ \forall g_w \text{ satisfying } \eqref{eq.g_w_constraint}.
		\end{gather*}
		In view of the S-lemma~\citep{polik2007survey}, the constraint of this minimization problem holds if and only if there exist $u$ and $\alpha\ge 0$ such that
		$Q_g(u, \gamma)-\alpha A_w\ge 0$.
		Using Schur complement~\citep{boyd1994linear}, the above matrix inequality can be transformed to the LMI in~\eqref{eq.finalLMI}.
		Minimizing the performance index $\gamma$ further gives the solution of~\eqref{eq.minimization_problem} and hence \textbf{P1}. 
	\end{proof}
        
\begin{remark}\label{rem:nonconservative}
      	The results presented in Theorem~\ref{thm.LMIopt} are non-conservative, i.e., the minimization problem~\eqref{eq.finalOpt},~\eqref{eq.finalLMI} is \textit{equivalent} to the min-max optimization problem in \textbf{P1}. Therefore, there exists a \textit{feasible} noise vector $w$ such that LQTE$(u, w)=\gamma^*$, where $\gamma^*$ is the minimizer of~\eqref{eq.finalOpt},~\eqref{eq.finalLMI}.
\end{remark}
        
\begin{remark}
  The proposed control design can easily be applied in a receding horizon fashion, in order to implement a data-driven predictive controller. In doing so, at each time $t$, one needs to update the output reference $r$, as well as recent input and output data $u_\mathrm{ini}$ and $y_\mathrm{ini}$ with the online data, solve~\eqref{eq.finalOpt},~\eqref{eq.finalLMI}, and apply only the first control input from the computed optimal control sequence $u^*$.
  Moreover, it can be shown that the resulting controller is equivalent to a robust model predictive controller (MPC) with bounded uncertainty on the initial state.
  As such, the stability of the resulting closed-loop system can be studied using the existing results on robust MPC.
\end{remark}
\vspace{-0.5cm}
	\section{Simulations}\label{subsec.Simulation}
	
	We consider the control of an unstable LTI system in the form of~\eqref{eq.LTI_system_ss} with the randomly selected matrices \vspace{-0.7cm}
	
	\begin{equation*}
	\begin{split}
	A &= \begin{bmatrix}
	0.8768 & 0.4147 & 0.0678 \\
	0.3934 & -0.6436 & -0.2961 \\
	-0.7907 & 0.7055 & 0.1587
	\end{bmatrix} \quad
	B = \begin{bmatrix}
	0.9567 \\
	0.1039 \\
	-0.2155
	\end{bmatrix} \quad
	C = \begin{bmatrix}
	0.4164 \\
	-0.7185 \\
	-0.9618
	\end{bmatrix}^\top
	\end{split} \quad D=0.
	\end{equation*}
	By solving (\ref{eq.finalOpt})-(\ref{eq.finalLMI}), we aim to calculate the optimal inputs $u^*$ and the resulting worst-case cost $\gamma^*$ to regulate the system to zero outputs. We justify the non-conservativeness of our algorithm by showing, through multiple noise realizations, that $\gamma^*$ is not an overestimate of the actual linear-quadratic tracking errors.
	
	With a random initial state, \textit{historical} input-output data of length $T_d=100$ are collected from this system with inputs sampled from a uniform distribution in the interval $[-1,1]$. We assume the exact order $n=3$ to be unknown and let $T_{\text{ini}} = 4$. Prior to the optimal control horizon, we measure the \textit{recent} data $\{u_\mathrm{ini},y_\mathrm{ini}\}$ of length $T_{\text{ini}}$ where the inputs are generated in the same way as those in historical data. Moreover, this recent data is corrupted by a noise trajectory $w$ verifying the quadratic constraint~\eqref{eq.yini-noiseconstraint} with  
	$\Phi_{11} = T_\mathrm{ini}p\epsilon,\enskip \Phi_{12}=\0,\enskip \Phi_{22} = -\I$,
	and $\epsilon=0.001$.
	
    In \textbf{P1}, we select $T_e=20$ and $r=\0$ to robustly regulate the output of the system to zero within a horizon of length $20$. The LMI minimization problem~\eqref{eq.finalOpt} is solved using Yalmip~\citep{Lofberg2004} and MOSEK \citep{mosek}. The optimal control sequence $u^*$ is tested with multiple compatible realizations of noise trajectories. In particular, we randomly select $100$ $g_w$ vectors verifying the quadratic constraint~\eqref{eq.g_w_constraint}. These realizations parameterize $100$ feasible trajectories of $w$, each verifying the quadratic constraint in~\eqref{eq.yini-noiseconstraint}. As shown in Figure~\ref{fig.LTI_outputs}, output trajectories quickly converge to a neighborhood of zero for all selected noise realizations. 
	
	
	\begin{figure}
    \begin{minipage}[b]{.5\linewidth}
        \centering
        \includegraphics[width=\textwidth]{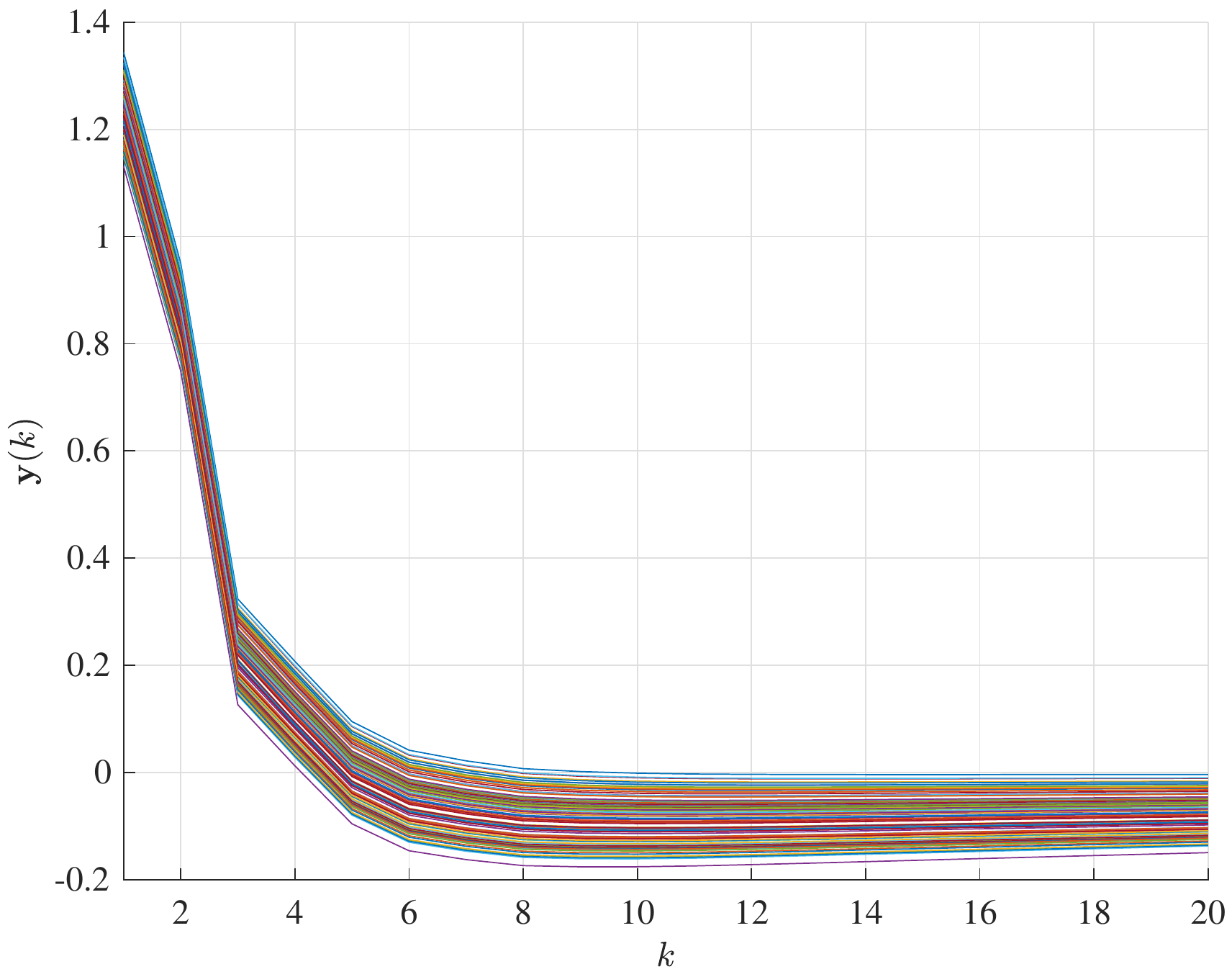}
        \caption{Closed-loop output trajectories for different noise realizations}	\label{fig.LTI_outputs}
      \end{minipage}%
      \vspace{-0.5cm}
    \hspace{.3cm}
      \begin{minipage}[b]{.5\linewidth}
        \centering
        \includegraphics[width=1\textwidth]{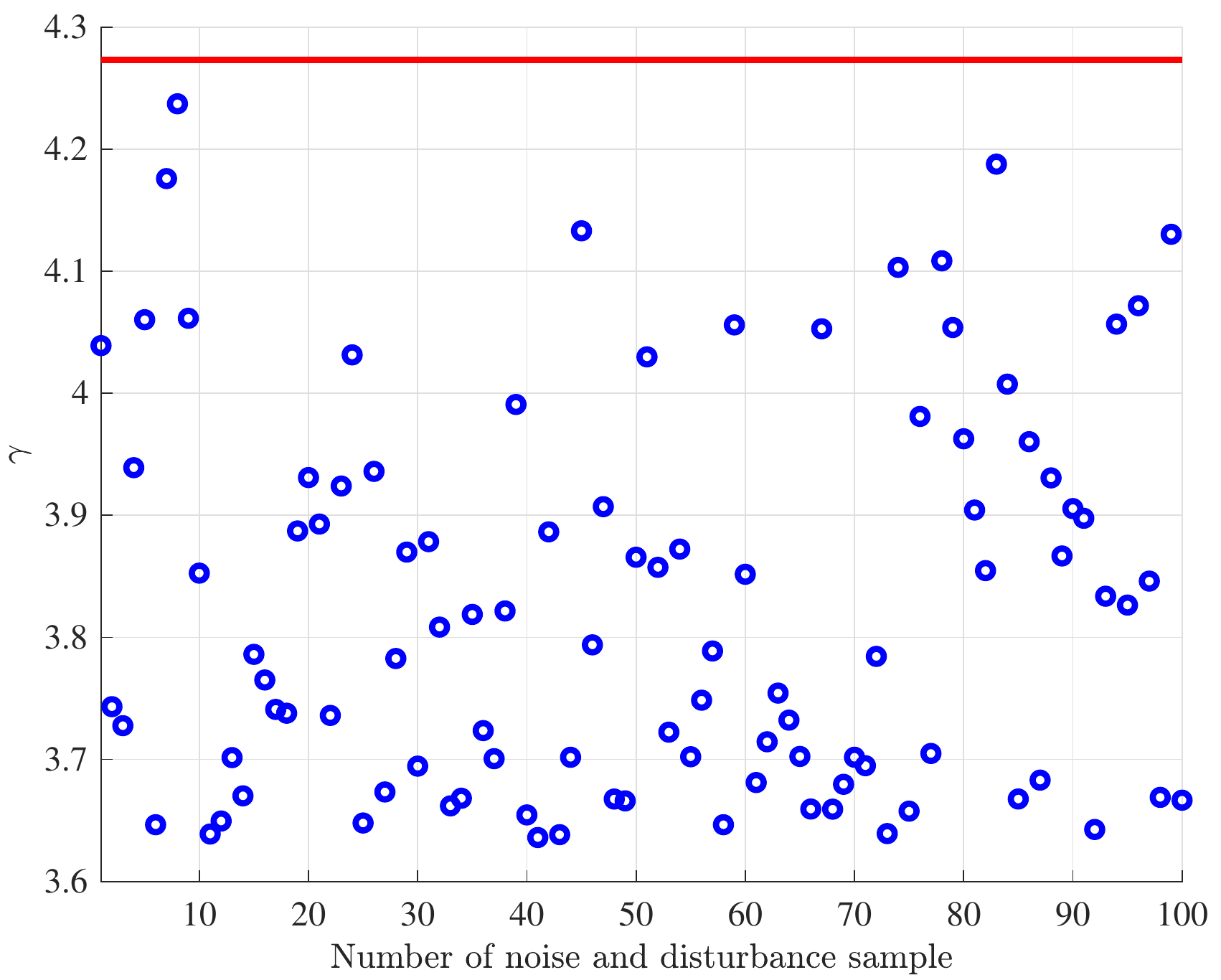}
        \caption{Costs for different noise realiza-\\tions compared with $\gamma^*$ (red line) }\label{fig.LTI_constraints}
      \end{minipage}
      \vspace{-0.5cm}
\end{figure}

	
As seen in Figure~\ref{fig.LTI_constraints}, the costs $\gamma=y^\top Q y+{u^*}^\top Ru^*$ of all the noise realizations (show in blue circles) are below the solution $\gamma^*$ to \eqref{eq.finalOpt} (shown in red line). Besides, one can spot a cost realization close to $\gamma^*$, which indicates that $\gamma^*$ is not a conservative estimate.

	\section{Conclusion}\label{sec.Conclusion}
	In this paper, we build on  data-dependent behavioral representations of linear systems to consider the case that the recent output data are noisy and solve the data driven robust optimal tracking control problem.	
	However, the proposed method assumes that in the data-dependent representation only recent data is noisy. Future work will be devoted to studying the impact of noise in the historical data.
\vspace{-0.5cm}
        \acks{This research is supported by the Swiss National Science Foundation under the NCCR Automation (grant agreement 51NF40\_180545) and the COFLEX project (grant agreement 200021 169906).}
        
	\bibliography{references}
	
\end{document}